\documentclass[a4paper,11pt]{article}

\usepackage[leqno]{amsmath}
\usepackage{amssymb,amsthm,ifthen}
\usepackage{url,enumitem}
\usepackage[pdftex]{graphicx}
\usepackage[colorinlistoftodos]{todonotes}
\usepackage[normalem]{ulem}
\usepackage{microtype}

\usepackage{yhmath}

\usepackage{tikz}
\usepackage{tikz-cd}
\usetikzlibrary{calc}
\usepgflibrary{shapes.geometric}
\usepgflibrary{shapes.misc}
\usetikzlibrary{positioning}
\usetikzlibrary{decorations}
\usetikzlibrary{arrows}

\usepackage[a4paper,left=1in,width=458pt,textheight=650pt,top=1.5in]{geometry}

\usepackage{bbm}

\usepackage{hyperref}

\input cyracc.def

\hypersetup{
pdftitle={Graded Frobenius cluster categories},
pdfauthor={Jan E. Grabowski and Matthew Pressland},
pdfstartview=FitH
}

\DeclareMathOperator{\fgmod}{mod}
\DeclareMathOperator{\fd}{fd}
\DeclareMathOperator{\CM}{CM}
\DeclareMathOperator{\Sub}{Sub}

\DeclareMathOperator{\GP}{GP}

\newcommand{\add}{\operatorname{add}}

\newcommand{\bigdsum}{\bigoplus}

\newcommand{\card}[1]{\lvert #1 \rvert}

\let\chisave\chi
\renewcommand{\chi}{{%
 \mathchoice{\raisebox{0.25ex}{$\displaystyle\chisave$}}
            {\raisebox{0.2ex}{$\textstyle\chisave$}}
            {\raisebox{0.2ex}{$\scriptstyle\chisave$}}
            {\raisebox{0.1ex}{$\scriptscriptstyle\chisave$}}}}

\newcommand{\complex}{\ensuremath \mathbb{C}}

\newcommand{\cross}{\times}
\newcommand{\curly}[1]{\ensuremath{\mathcal{#1}}}

\newcommand{\defeq}{:=}
\let\degsave\deg
\renewcommand{\deg}{\underline{\degsave}}

\newcommand{\dimvec}{\ensuremath \mbox{\underline{dim}}}

\newcommand{\dsum}{\ensuremath{ \oplus}}

\newcommand{\End}[2]{\ensuremath \operatorname{End}_{#1}(#2)}
\renewcommand{\epsilon}{\varepsilon}
\newcommand{\Ext}[4]{\ensuremath \operatorname{Ext}^{#1}_{#2}(#3,#4)}

\newcommand{\Hom}[3]{\ensuremath \operatorname{Hom}_{#1}(#2,#3)}

\newcommand{\ind}[2]{\text{\underline{ind}}_{#1}(#2)}

\newcommand{\integ}{\ensuremath{\mathbb{Z}}}

\newcommand{\ip}[2]{\ensuremath \lgen\;\!#1,#2\;\!\rgen}

\newcommand{\iso}{\ensuremath \cong}

\newcommand{\lgen}{\ensuremath \mathopen{<}} 

\newcommand{\op}[1]{{#1}^{\mbox{\scriptsize \textup{op}}}}

\renewcommand{\phi}{\varphi}

\newcommand{\rad}{\mathop{\mathrm{rad}}}

\newcommand{\rgen}{\ensuremath \mathclose{>}}

\renewcommand{\leq}{\leqslant}
\renewcommand{\geq}{\geqslant}

\newcommand{\powser}[2]{#1[\hspace{-0.1em}[#2]\hspace{-0.1em}]}

\newcommand{\abgp}{\mathbb{A}}

\newcommand{\cf}{cf.\ }

\theoremstyle{plain}
\newtheorem{theorem}{Theorem}[section]
\newtheorem*{theorem*}{Theorem}
\newtheorem{proposition}[theorem]{Proposition}
\newtheorem{lemma}[theorem]{Lemma}

\theoremstyle{definition}
\newtheorem{definition}[theorem]{Definition}
\newtheorem*{question*}{Question}
\theoremstyle{remark}
\newtheorem{remark}[theorem]{Remark}
\newtheorem{example}[theorem]{Example}
\newtheorem*{example*}{Example}
\newtheorem*{examplectd*}{Example (continued)}

\title{Graded Frobenius cluster categories}
\author{Jan E. Grabowski\footnotemark[2] 
\\ \small{\textit{Department of Mathematics and Statistics, Lancaster University,}}
\\ \small{\textit{Lancaster, LA1 4YF, United Kingdom}}
\and Matthew Pressland\footnotemark[3]
\\ \small{\textit{Institut f\"ur Algebra und Zahlentheorie, Universit\"at Stuttgart,}}
\\ \small{\textit{Pfaffenwaldring 57, 70569 Stuttgart, Germany}}
}
\date{30th October 2017}

\begin{document}

\maketitle

\renewcommand{\thefootnote}{\fnsymbol{footnote}}
\footnotetext[2]{Email: \url{j.grabowski@lancaster.ac.uk}.  Website: \url{http://www.maths.lancs.ac.uk/~grabowsj/}}
\footnotetext[3]{Email: \url{mdpressland@bath.edu}.  Website: \url{http://www.iaz.uni-stuttgart.de/LstAGeoAlg/Pressland/}}
\renewcommand{\thefootnote}{\arabic{footnote}}
\setcounter{footnote}{0}

\begin{abstract} Recently the first author studied multi-gradings for generalised cluster categories, these being 2-Calabi--Yau triangulated categories with a choice of cluster-tilting object.  The grading on the category corresponds to a grading on the cluster algebra without coefficients categorified by the cluster category and hence knowledge of one of these structures can help us study the other.

In this work, we extend the above to certain Frobenius categories that categorify cluster algebras with coefficients.  We interpret the grading K-theoretically and prove similar results to the triangulated case, in particular obtaining that degrees are additive on exact sequences.

We show that the categories of Buan, Iyama, Reiten and Scott, some of which were used by Gei\ss, Leclerc and Schr\"oer to categorify cells in partial flag varieties, and those of Jensen, King and Su, categorifying Grassmannians, are examples of graded Frobenius cluster categories.

\vspace{1em}

\end{abstract}

\section{Introduction}

Gradings for cluster algebras have been introduced in various ways by a number of authors and for a number of purposes.  The evolution of the notion started with the foundational work of Fomin and Zelevinsky \cite{FZ-CA1}, who consider $\integ^{n}$-gradings where $n$ is precisely the rank of the cluster algebra.  Shortly afterwards, in the course of considering Poisson structures compatible with cluster algebras, Gekhtman, Shapiro and Vainshtein \cite{GSV-CA-Poisson} gave a definition of a toric action on a cluster algebra, which dualises to that of a $\integ^{m}$-grading, where $m$ can now be arbitrary.  

In \cite{GradedCAs} the first author examined the natural starting case of finite type cluster algebras without coefficients.  A complete classification of the integer multi-gradings that occur was given and it was observed that the gradings so obtained were all \emph{balanced}, that is, there exist bijections between the set of variables of degree $\underline{d}$ and those of degree $-\underline{d}$.

This phenomenon was explained by means of graded generalised cluster categories, where---following \cite{DominguezGeiss}---by generalised cluster category we mean a 2-Calabi--Yau triangulated category $\curly{C}$ with a basic cluster-tilting object $T$.  The definition made in \cite{GradedCAs} associates an integer vector (the multi-degree) to an object in the category in such a way that the vectors are additive on distinguished triangles and transform naturally under mutation.  This is done via the key fact that every object in a generalised cluster category has a well-defined associated integer vector-valued datum called the index with respect to $T$; in order to satisfy the aforementioned two properties, degrees are necessarily linear functions of the index.

The categorical approach has the advantage that it encapsulates the global cluster combinatorics, or more accurately the set of indices does.  Another consequence is an explanation for the observed balanced gradings in finite type: the auto-equivalence of the cluster category given by the shift functor induces an automorphism of the set of cluster variables that reverses signs of degrees.  Hence any cluster algebra admitting a (triangulated) cluster categorification necessarily has all its gradings being balanced (providing the set of reachable rigid indecomposable objects, which is in bijection with the set of cluster variables, is closed under the shift functor). This is the case for finite type or, more generally, acyclic cluster algebras having no coefficients.

Our main goal is to provide a version of the above in the Frobenius, i.e.\ exact category, setting, similarly to the triangulated one.  A Frobenius category is exact with enough projective objects and enough injective objects, and these classes of objects coincide.  From work of Fu and Keller \cite{FuKeller} and the second author \cite{Pressland}, we have a definition of a Frobenius cluster category and objects in such a category also have indices.  

From this we may proceed along similar lines to \cite{GradedCAs} to define gradings and degrees, except that we elect to work (a) over an arbitrary abelian group $\mathbb{A}$ and (b) in a more basis-free way by working K-theoretically and with the associated Euler form.  We prove the foundational properties of gradings for Frobenius cluster categories: that degrees are compatible with taking the cluster character, that they are additive on exact sequences and that they are compatible with mutation.

Furthermore, we prove an analogue of a result of Palu \cite{Palu-Groth-gp} in which we show that the space of gradings for a graded Frobenius cluster category $\curly{E}$ is closely related to the Grothendieck group, namely that the former is isomorphic to $\Hom{\integ}{\mathrm{K}_{0}(\curly{E})}{\abgp}$.  This enables one to show that some categorical datum is a grading by seeing that it respects exact sequences, and conversely that from the cluster algebra categorified by $\curly{E}$ we may deduce information about $\mathrm{K}_{0}(\curly{E})$.  We exhibit this on examples, notably the categories of Buan, Iyama, Reiten and Scott \cite{BIRS1} corresponding to Weyl group elements, also studied by Gei\ss, Leclerc and Schr\"oer \cite{GLS-PFV} in the context of categorifying cells in partial flag varieties.

The homogeneous coordinate rings of Grassmannians are an example of particular importance in this area.  They admit a graded cluster algebra structure but beyond the small number of cases when this structure is of finite type, little is known about the cluster variables.  A first step towards a better understanding is to describe how the degrees of the cluster variables are distributed: are the degrees unbounded? does every natural number occur as a degree? are there finitely many or infinitely many variables in each occurring degree?  By using the Frobenius cluster categorification of Jensen, King and Su \cite{JKS} and the grading framework here, we can hope to begin to examine these questions.

\subsection*{Acknowledgements}

The work herein was begun during a research visit of the second author to Lancaster University, funded by EPSRC grant EP/M001148/1.  The second author would like to thank Paul Balmer, Andrew Hubery, Alastair King, Sondre Kvamme and Pierre-Guy Plamondon for useful conversations, and acknowledge financial support from Sonderforschungsbereich 701 at Universit\"{a}t Bielefeld, and from the Max-Planck-Gesellschaft.

\section{Preliminaries}
\label{preliminaries}

The construction of a cluster algebra of geometric type from an initial seed $(\underline{x},B)$, due to Fomin and Zelevinsky \cite{FZ-CA1}, is now well-known.  Here $\underline{x}=(x_1,\dotsc,x_n)$ is a transcendence base for a certain field of fractions of a polynomial ring, called a cluster, and $B$ is an $n\times r$ integer matrix whose uppermost $r\times r$ submatrix (the principal part of $B$) is skew-symmetrisable. If the principal part of $B$ is skew-symmetric, then $B$ is often replaced by the (ice) quiver $Q=Q(B)$ it defines in the natural way.  

We refer the reader who is unfamiliar with this construction to the survey of Keller \cite{Keller-CASurvey} and the books of Marsh \cite{Marsh-book} and of Gekhtman, Shapiro and Vainshtein \cite{GSV-Book} for an introduction to the topic and summaries of the main related results in this area.

We set some notation for later use.  For each index $1\leq k\leq r$, set
\begin{align*}
\underline{b}_{k}^{+} & = -\underline{\boldsymbol{e}}_{k}+\sum_{b_{ik}>0}b_{ik}\underline{\boldsymbol{e}}_{i} \qquad \text{and} \\
\underline{b}_{k}^{-} & = -\underline{\boldsymbol{e}}_{k}-\sum_{b_{ik}<0}b_{ik}\underline{\boldsymbol{e}}_{i},
\end{align*}
where the vector $\underline{\boldsymbol{e}}_{i}\in \integ^{n}$ ($n$ being the number of rows of $B$) is the $i$th standard basis vector.  Note that the $k$th column of $B$ may be recovered as $B_{k}=\underline{b}_{k}^{+}-\underline{b}_{k}^{-}$.

Then for $1\leq k\leq r$, the exchange relation for mutation of the seed $(\underline{x},B)$ in the direction $k$ is given by
\[ x_{k}^{\prime}=\underline{x}^{\underline{b}_{k}^{+}}+\underline{x}^{\underline{b}_{k}^{-}}, \]
where for $\underline{a}=(a_{1},\dotsc ,a_{n})$ we set
\[ \underline{x}^{\underline{a}} = \prod_{i=1}^{n} x_{i}^{a_{i}}. \]

If $(\underline{x},B)$ is a seed, we call the elements of $\underline{x}$ cluster variables. The variables $x_1,\dotsc,x_r$ are called mutable, and $x_{r+1},\dots,x_n$ (which appear in the cluster of every seed related to $(\underline{x},B)$ by mutations) are called frozen. Note that while some authors do not consider frozen variables to be cluster variables, it will be convenient for us to do so. We will sometimes also refer to the indices $1,\dotsc,r$ and $r+1,\dotsc,n$ as mutable and frozen respectively.

Throughout, for simplicity, we will assume that all algebras and categories are defined over $\mathbb{C}$.  All modules are left modules. For a Noetherian algebra $A$, we denote the abelian category of finitely generated $A$-modules by $\fgmod{A}$. If $B$ is a matrix, we denote its transpose by $B^{t}$.

\subsection{Graded seeds, cluster algebras and cluster categories}\label{s:gradedCAs}

Let $\abgp$ be an abelian group. The natural definition for an $\abgp$-graded seed is as follows.

\begin{definition} A multi-graded seed is a triple $(\underline{x},B,G)$ such that
\begin{enumerate}[label=(\alph*)]
\item $(\underline{x}=(x_{1},\dotsc ,x_{n}),B)$ is a seed, and
\item $G\in\abgp^n$, thought of as a column vector, satisfies $B^{t}G=0$.
\end{enumerate}
The matrix multiplication in (b) makes sense since $\abgp$ is a $\integ$-module. This is most transparent when $\abgp=\integ^m$, so that $G$ is an $n\times m$ integer matrix.
\end{definition}

From now on, unless we particularly wish to emphasise $\abgp$, we will drop it from the notation and simply use the term ``graded''.

The above data defines $\deg_{G}(x_{i})=G_{i}\in\abgp$ (the $i$th component of $G$) and this can be extended to rational expressions in the generators $x_{i}$ in the obvious way. Condition (b) ensures that for each $1\leq k\leq r$, we have $\underline{b}_k^+\cdot G=\underline{b}_k^-\cdot G$, making sense of these dot products via the $\integ$-module structure of $\abgp$, so every exchange relation is homogeneous, and
\[G_k':=\deg(x_k')=\underline{b}_k^+\cdot G-G_k=\underline{b}_k^-\cdot G-G_k.\]
Thus we can also mutate our grading, and repeated mutation propagates a grading on an initial seed to every cluster variable and to the associated cluster algebra.  Hence we obtain the following well-known result, given in various forms in the literature.

\begin{proposition}\label{p:gradedCA} The cluster algebra $\curly{A}(\underline{x},B,G)$ associated to an initial graded seed $(\underline{x},B,G)$ is an $\abgp$-graded algebra. Every cluster variable of $\curly{A}(\underline{x},B,G)$ is homogeneous with respect to this grading.  \qed
\end{proposition}

We refer the reader to \cite{GradedCAs} for a more detailed discussion of the above and further results regarding the existence of gradings, relationships between gradings and a study of $\integ$-gradings for cluster algebras of finite type with no coefficients.

\subsection{Graded triangulated cluster categories}\label{s:graded-triang-cluster-categories}

Our interest here is in generalising the categorical parts of \cite{GradedCAs}, which refer to models of cluster algebras without frozen variables given by $2$-Calabi--Yau triangulated categories, and explain how to interpret the data of a grading on the cluster algebra at this categorical level. Our main goal is to provide a similar theory for stably $2$-Calabi--Yau Frobenius categories, which may be used to model cluster algebras that do have frozen variables.

In order to motivate what will follow for the Frobenius setting, we give the key definitions and statements from the triangulated case, without proofs as these may be found in \cite{GradedCAs}.

\begin{definition}[{\cite{DominguezGeiss}}] Let $\curly{C}$ be a triangulated 2-Calabi--Yau category with suspension functor $\Sigma$ and let $T\in \curly{C}$ be a basic cluster-tilting object.  We will call the pair $(\curly{C},T)$ a generalised cluster category.
\end{definition}

Write $T=T_{1}\dsum \dotsm \dsum T_{r}$.  Setting $\Lambda=\op{\End{\curly{C}}{T}}$, the functor\footnote{This functor is replaced by $E=F\Sigma$ in \cite{DominguezGeiss}, \cite{GradedCAs}; we use $F$ here, as in \cite{FuKeller}, for greater compatibility with the Frobenius case.} $F=\curly{C}(T,-)\colon \curly{C} \to \fgmod{\Lambda}$ induces an equivalence $\curly{C}/\text{add}(\Sigma T)\to\fgmod{\Lambda}$.  We may also define an exchange matrix associated to $T$ by
\[ (B_{T})_{ij}=\dim \text{Ext}_{\Lambda}^{1}(S_{i},S_{j})-\dim \text{Ext}_{\Lambda}^{1}(S_{j},S_{i}). \]
Here the $S_{i}=FT_{i}/\rad FT_{i}$, $i=1,\dotsc ,r$ are the simple $\Lambda$-modules. Thus, if the Gabriel quiver of the algebra $\Lambda$ has no loops or $2$-cycles, $B_T$ is its corresponding skew-symmetric matrix.

For each $X\in \curly{C}$ there exists a distinguished triangle
\[ \bigdsum_{i=1}^{r} T_{i}^{m(i,X)} \to \bigdsum_{i=1}^{r} T_{i}^{p(i,X)} \to X \to \Sigma \left( \bigdsum_{i=1}^{r} T_{i}^{m(i,X)} \right) \]

Define the index of $X$ with respect to $T$, denoted $\ind{T}{X}$, to be the integer vector with $\ind{T}{X}_{i}=p(i,X)-m(i,X)$.  By \cite[\S 2.1]{Palu}, $\ind{T}{X}$ is well-defined and we have a cluster character
\begin{align*} C_{?}^{T}\colon \text{Obj}(\curly{C}) &\to \mathbb{C}[x_{1}^{\pm 1},\dotsc ,x_{r}^{\pm 1}] \\
 X & \mapsto \underline{x}^{\ind{T}{X}}\sum_{\underline{e}} \chi(\mathrm{Gr}_{\underline{e}}(F\Sigma X))\underline{x}^{B_{T}\cdot\underline{e}}
\end{align*}
Here $\mathrm{Gr}_{\underline{e}}(F\Sigma X)$ is the quiver Grassmannian of $\Lambda$-submodules of $F\Sigma X$ of dimension vector $\underline{e}$ and $\chi$ is the topological Euler characteristic.  We use the same monomial notation $\underline{x}^{\underline{a}}$ as previously.

We recall that for any cluster-tilting object $U$ of $\curly{C}$, and any indecomposable summand $U_k$ of $U$, there are non-split exchange triangles
\[ U_{k}^{*} \to M \to U_{k} \to \Sigma U_{k}^{*} \qquad \text{and} \qquad U_{k} \to M' \to U_{k}^{*} \to \Sigma U_{k} \]
with $M,M' \in \add(U)$, that glue together to form an Auslander--Reiten $4$-angle
\[U_k\to M'\to M\to U_k\]
in $\curly{C}$ \cite[Definition~3.8]{IyamaYoshino}. If the quiver of $\op{\End{\curly{C}}{U}}$ has no loops or $2$-cycles incident with the vertex corresponding to $U_k$, then $M,M'\in\add(U/U_k)$ and $U^{*}=(U/U_{k})\dsum U_{k}^{*}$ is again cluster-tilting. In the generality of our setting, these results are all due to Iyama and Yoshino \cite{IyamaYoshino}.

The natural definition of a graded generalised cluster category is then the following.

\begin{definition}[{\cite[Definition~5.2]{GradedCAs}}]\label{d:graded-gen-cl-cat} Let $(\curly{C},T)$ be a generalised cluster category and let $G\in\abgp^r$ such that $B_{T}G=0$.  We call the tuple $(\curly{C},T,G)$ a graded generalised cluster category.
\end{definition}

Note that, in this context, $B_{T}$ is square and skew-symmetric, so we may suppress taking the transpose in the equation $B_{T}G=0$.

\begin{definition}[{\cite[Definition~5.3]{GradedCAs}}]\label{d:degree-graded-gen-cl-cat} Let $(\curly{C},T,G)$ be a graded generalised cluster category.  For any $X\in \curly{C}$, we define $\deg_{G}(X)=\ind{T}{X}\cdot G$.\end{definition}

The main results about graded generalised cluster categories are summarised in the following Proposition, the most significant of these being \ref{p:prop-of-gen-cc-additive-on-triang}. The proofs in \cite{GradedCAs} are given for $\abgp=\integ^m$, but remain valid in the more general setting.

\begin{proposition}[{\cite[\S5]{GradedCAs}}]\label{p:prop-of-gen-cc} Let $(\curly{C},T,G)$ be a graded generalised cluster category.
\begin{enumerate}[label=(\roman*)]
\item Let $\mathbb{C}[x_{1}^{\pm 1},\dotsc ,x_{r}^{\pm 1}]$ be $\abgp$-graded by $\deg_{G}(x_{i})=G_{i}$ (the $i$th component of $G$). Then for all $X \in \curly{C}$, the cluster character $C_{X}^{T}\in \mathbb{C}[x_{1}^{\pm 1},\dotsc ,x_{r}^{\pm 1}]$ is homogeneous of degree $\deg_G(X)$.
\item\label{p:prop-of-gen-cc-additive-on-triang} For any distinguished triangle $X\to Y \to Z \to \Sigma X$ of $\curly{C}$, we have
\[ \deg_{G}(Y)=\deg_{G}(X)+\deg_{G}(Z).\]
\item\label{p:prop-of-gen-cc-mutation} The degree $\deg_{G}$ is compatible with mutation in the sense that for every cluster-tilting object $U$ of $\curly{C}$ with indecomposable summand $U_{k}$ we have
\[ \deg_{G}(U_{k}^{*})=\deg_{G}(M)-\deg_{G}(U_{k})=\deg_{G}(M')-\deg_{G}(U_{k}), \]
where $U_{k}^{*}$, $M$ and $M'$ are as in the above description of exchange triangles in $\curly{C}$.
\item\label{p:prop-of-gen-cc-Groth-gp} The space of gradings for a generalised cluster category $(\curly{C},T)$ may be identified with $\Hom{\integ}{\mathrm{K}_0(\curly{C})}{\abgp}$, where $\mathrm{K}_0(\curly{C})$ is the Grothendieck group of $\curly{C}$ as a triangulated category.\footnote{This statement corrects \cite[Proposition~5.5]{GradedCAs} for the case $\abgp=\integ$, which replaces $\Hom{\integ}{\mathrm{K}_0(\curly{C})}{\integ}$ by $\mathrm{K}_0(\curly{C})$ itself. The proof given in \cite{GradedCAs} proves the statement given here for an arbitrary abelian group essentially without modification. An example of $\curly{C}$ for which $\mathrm{K}_0(\curly{C})$ and $\Hom{\integ}{\mathrm{K}_0(\curly{C})}{\integ}$ are non-isomorphic is provided by \cite[Thm.~1.3]{BKL}.}
\item For each $X\in \curly{C}$, $\deg_{G}(\Sigma X)=-\deg_{G}(X)$.  That is, for each $d\in \abgp$, the shift automorphism $\Sigma$ on $\curly{C}$ induces a bijection between the objects of $\curly{C}$ of degree $d$ and those of degree $-d$. \qed
\end{enumerate}
\end{proposition}

Part~\ref{p:prop-of-gen-cc-mutation} of the preceding proposition shows how to mutate the data of $G$ when mutating the cluster-tilting object $T$, to obtain a new grading vector compatible with the exchange matrix of the new cluster-tilting object, defining the same grading on the cluster algebra. 

However, we may obtain an even stronger conclusion from part~\ref{p:prop-of-gen-cc-Groth-gp}, since this provides a ``base-point free'' definition of a grading, depending only on the category $\curly{C}$ and not on the cluster-tilting object $T$. Read differently, this shows that if $(\curly{C},T,G)$ is a graded generalised cluster category, then for any cluster-tilting object $T'\in\curly{C}$, there is a unique $G'\in\abgp^r$ such that $(\curly{C},T',G')$ is a graded generalised cluster category and $\deg_G(X)=\deg_{G'}(X)$ for all $X\in\curly{C}$. We will explain this in more detail below in the case of Frobenius categories.

If $\curly{H}$ is the category of coherent sheaves on a weighted projective line with all weights odd, then the Grothendieck group of the cluster category $\curly{C}$ of $\curly{H}$ is a non-zero quotient of $\integ_2\oplus\integ_2$ \cite[Theorem~1.3]{BKL}. (If one only imposes relations coming from triangles obtained by projecting triangles of the derived category of $\curly{C}$ to $\curly{H}$, then one obtains exactly $\integ_2\oplus\integ_2$ \cite[Proposition~3.7(ii)]{BKL}, but $\curly{C}$ may have more triangles than these.) By part~\ref{p:prop-of-gen-cc-Groth-gp} of the preceding proposition, this cluster category $\curly{C}$ admits no $\integ$-gradings, but does admit $\integ_2$-gradings. In fact \cite[Proposition 3.10(ii)]{BKL}, any such grading is a linear combination of the functions giving the degree and rank of a sheaf modulo $2$.

Part (v) of Proposition~\ref{p:prop-of-gen-cc} shows that for cluster algebras admitting a categorification by a generalised cluster category $(\curly{C},T)$ such that the mutation class of $T$ is closed under the shift functor $\Sigma$, all gradings must be balanced, meaning that for any $d\in\abgp$, the cluster variables of degree $d$ are in bijection with those of degree $-d$.

If $Q$ admits a nondegenerate Jacobi-finite potential $W$, then the corresponding cluster algebra is categorified by the Amiot cluster category $\curly{C}_{Q,W}$, which has a cluster-tilting object $T$ whose endomorphism algebra is the Jacobian algebra of $(Q,W)$ \cite{Amiot}. If $Q$ admits a maximal green sequence, then it provides a sequence of mutations from $T$ to $\Sigma T$ in $\curly{C}_{Q,W}$, so the mutation class of $T$ is closed under $\Sigma$ \cite[Proposition~5.17]{Keller-QD}. It follows that all gradings of the cluster algebra associated to $Q$ are balanced. All of these assumptions hold, for example, when $Q$ is a finite acyclic quiver (so $W=0$); for the statement about maximal green sequences, see Br\"{u}stle, Dupont and P\'{e}rotin \cite[Lemma~2.20]{BDP}.

Conversely, we can use gradings to show that certain cluster algebras cannot admit a categorification as above. For example, the Markov cluster algebra, all of whose exchange matrices are given by
\[B=\begin{pmatrix}0&2&-2\\-2&0&2\\2&-2&0\end{pmatrix}\]
or its negative, admits the grading $(1,1,1)$. This is an integer grading under which all cluster variables have strictly positive degrees, so it is not balanced. While the Markov quiver associated to $B$ has a non-degenerate potential for which the resulting (completed) Jacobian algebra is finite dimensional, and thus has an associated Amiot cluster category $\curly{C}$, this category has exactly two mutation classes of cluster-tilting objects. (One can also realise this Jacobian algebra as that coming from a tagged triangulation of the once-punctured torus; such triangulations can include tagged arcs or not, but it is not possible to mutate a triangulation without tagged arcs into one with tagged arcs, giving another explanation for the existence of these two mutation classes.) The shift functor on $\curly{C}$ takes rigid indecomposable objects appearing as summands in one mutation class (which correspond to cluster variables) to rigid indecomposables from the other class (which do not), allowing the existence of a non-balanced grading on the cluster algebra.

It has been shown by Ladkani that many of these properties hold more generally for quivers arising from triangulations of punctured surfaces \cite{Ladkani}.

\section{Graded Frobenius cluster categories}

In this section, we provide the main technical underpinnings for the Frobenius version of the above theory, in which we consider exact categories rather than triangulated ones. Background on exact categories, and homological algebra in them, can be found in B\"{u}hler's survey \cite{Buehler}.

An exact category $\curly{E}$ is called a Frobenius category if it has enough projective objects and enough injective objects, and these two classes of objects coincide. A typical example of such a category is the category of finite dimensional modules over a finite dimensional self-injective algebra. More generally, if $B$ is a Noetherian algebra with finite left and right injective dimension as a module over itself (otherwise known as an Iwanaga--Gorenstein algebra), the category
\[\operatorname{GP}(B)=\{X\in\fgmod{B}:\Ext{i}{B}{X}{B}=0,\ i>0\},\]
is Frobenius \cite{Buchweitz}. (Here $\operatorname{GP}(B)$ is equipped with the exact structure in which the exact sequences are precisely those  that are exact when considered in the abelian category $\fgmod{B}$.)  The initials ``GP'' are chosen for ``Gorenstein projective''.

Given a Frobenius category $\curly{E}$, its stable category $\underline{\curly{E}}$ is formed by taking the quotient of $\curly{E}$ by the ideal of morphisms factoring through a projective-injective object. By a famous result of Happel \cite[Theorem~2.6]{Happelbook}, $\underline{\curly{E}}$ is a triangulated category with shift functor $\Omega^{-1}$, where $\Omega^{-1}X$ is defined by the existence of an exact sequence
\[0\to X\to Q \to\Omega^{-1}X\to0\]
in which $Q$ is injective. The distinguished triangles of $\underline{\curly{E}}$ are isomorphic to those of the form
\[X\to Y\to Z\to\Omega^{-1}X\]
where
\[0\to X\to Y\to Z\to 0\]
is a short exact sequence in $\curly{E}$.

\begin{definition}
A Frobenius category $\curly{E}$ is stably $2$-Calabi--Yau if the stable category $\underline{\curly{E}}$ is Hom-finite and there is a functorial duality
\[\mathrm{D}\Ext{1}{\curly{E}}{X}{Y}=\Ext{1}{\curly{E}}{Y}{X}\]
for all $X,Y\in\curly{E}$.
\end{definition}

\begin{remark} The above definition is somewhat slick---it is equivalent to requiring that $\underline{\curly{E}}$ is \linebreak $2$-Calabi--Yau as a triangulated category (that is, that $\underline{\curly{E}}$ is Hom-finite and $\Omega^{-2}$ is a Serre functor), as one might expect.
\end{remark}

Let $\curly{E}$ be a stably $2$-Calabi--Yau Frobenius category. If $U$ is cluster-tilting in $\curly{E}$, then it is also cluster-tilting in the $2$-Calabi--Yau triangulated category $\curly{E}$, and a summand $U_k$ of $U$ is indecomposable in $\underline{\curly{E}}$ if and only if it is indecomposable and non-projective in $\curly{E}$. Thus for any cluster-tilting object $U$ of $\curly{E}$ and for any non-projective indecomposable summand $U_{k}$ of $U$, we can lift the exchange triangles involving $U_k$ from $\underline{\curly{E}}$ to $\curly{E}$, and obtain exchange sequences
\[0\to U_{k}^{*} \to M \to U_{k} \to 0 \qquad \text{and} \qquad 0\to U_{k} \to M' \to U_{k}^{*} \to 0 \]
with $M,M' \in\add{(U)}$. If the quiver of $\op{\End{\curly{E}}{U}}$ has no loops or $2$-cycles incident with the vertex corresponding to $U_k$, then $U_k'=U/U_k\oplus U_k^*$ is again cluster-tilting, just as in the triangulated case.

Fu and Keller \cite{FuKeller} give the following definition of a cluster character on a stably $2$-Calabi--Yau Frobenius category.

\begin{definition}[{\cite[Definition~3.1]{FuKeller}}]
Let $\curly{E}$ be a stably $2$-Calabi--Yau Frobenius category, and let $R$ be a commutative ring. A cluster character on $\curly{E}$ is a map $\phi$ on the set of objects of $\curly{E}$, taking values in $R$, such that
\begin{itemize}
\item[(i)]if $M\iso M'$ then $\phi_{M}=\phi_{M'}$,
\item[(ii)]$\phi_{M\dsum N}=\phi_{M}\phi_{N}$, and
\item[(iii)]if $\dim\Ext{1}{\curly{E}}{M}{N}=1$ (equivalently, $\dim\Ext{1}{\curly{E}}{N}{M}=1$), and
\begin{align*}
&0\to M\to X\to N\to 0,\\
&0\to N\to Y\to M\to 0
\end{align*}
are non-split sequences, then
\[\phi_{M}\phi_{N}=\phi_{X}+\phi_{Y}.\]
\end{itemize}
\end{definition}

Let $\curly{E}$ be a stably $2$-Calabi--Yau Frobenius category, and assume there exists a cluster-tilting object $T\in\curly{E}$. Assume without loss of generality that $T$ is basic, and let $T=\bigdsum_{i=1}^nT_i$ be a decomposition of $T$ into pairwise non-isomorphic indecomposable summands. We number the summands so that $T_i$ is projective if and only if $r<i\leq n$. Let $\Lambda=\op{\End{\curly{E}}{T}}$, and $\underline{\Lambda}=\op{\End{\underline{\curly{E}}}{T}}=\Lambda/\Lambda e\Lambda$, where $e$ is the idempotent given by projection onto the maximal projective-injective summand $\bigoplus_{i=r+1}^nT_i$ of $T$. 

We assume that $\Lambda$ is Noetherian, as with this assumption the forms discussed below will be well-defined. The examples that concern us later will have Noetherian $\Lambda$, but we acknowledge that this assumption is somewhat unsatisfactory, given that it is often difficult to establish. 

Fu and Keller \cite{FuKeller} show that such a $T$ determines a cluster character on $\curly{E}$, as we now explain; while the results of \cite{FuKeller} are stated in the case that $\curly{E}$ is Hom-finite, the assumption that $\Lambda$ is Noetherian is sufficient providing one is careful to appropriately distinguish between the two Grothendieck groups $\mathrm{K}_0(\fgmod{\Lambda})$ and $\mathrm{K}_0(\fd{\Lambda})$ of finitely generated and finite dimensional $\Lambda$-modules respectively.

We write
\begin{align*}
F&=\Hom{\curly{E}}{T}{-}\colon\curly{E}\to\fgmod{\Lambda},\\
E&=\Ext{1}{\curly{E}}{T}{-}\colon\curly{E}\to\fgmod{\Lambda}.
\end{align*}
Note that $E$ may also be expressed as $\Hom{\underline{\curly{E}}}{T}{\Omega^{-1}(-)}$, meaning it takes values in $\fgmod{\underline{\Lambda}}$. For $M\in\fgmod{\Lambda}$ and $N\in\fd{\Lambda}$, we write
\begin{align*}
\ip{M}{N}_1&=\dim\Hom{\Lambda}{M}{N}-\dim\Ext{1}{\Lambda}{M}{N},\\
\ip{M}{N}_3&=\dim\Hom{\Lambda}{M}{N}-\dim\Ext{1}{\Lambda}{M}{N}+\dim\Ext{2}{\Lambda}{M}{N}-\dim\Ext{3}{\Lambda}{M}{N}.
\end{align*}

The algebra $\underline{\Lambda}=\op{\End{\underline{\curly{E}}}{T}}$ is finite dimensional since $\underline{\curly{E}}$ is Hom-finite, so $\fgmod{\underline{\Lambda}}\subseteq\fd\Lambda$. Fu and Keller show \cite[Proposition~3.2]{FuKeller} that if $M\in\fgmod{\underline{\Lambda}}$, then $\ip{M}{N}_3$ depends only on the dimension vector $(\dim\Hom{\Lambda}{P_i}{M})_{i=1}^n$, where the $P_i=FT_i$ are a complete set of indecomposable projective $\Lambda$-modules. Thus if $v\in\integ^r$, we define
\[\ip{v}{N}_3:=\ip{M}{N}_3\]
for any $M\in\fgmod{\underline{\Lambda}}$ with dimension vector $v$.

Let $R=\mathbb{C}[x_1^{\pm1},\dotsc,x_n^{\pm1}]$ be the ring of Laurent polynomials in $x_1,\dotsc,x_n$. Define a map $X\to C^T_{X}$ on objects of $\curly{E}$, taking values in $R$, via the formula

\[C^T_{X}=\prod_{i=1}^nx_i^{\ip{FX}{S_i}_1}\sum_{v\in\integ^r}\chi(\text{Gr}_v(EX))\prod_{i=1}^nx_i^{-\ip{v}{S_i}_3}.\]

Here, as before, $\text{Gr}_v(EX)$ denotes the projective variety of submodules of $EX$ with dimension vector $v$, and $\chi(\text{Gr}_v(EX))$ denotes its Euler characteristic. The modules $S_i=FT_i/\rad FT_i$ are the simple tops of the projective modules $P_i$. By \cite[Theorem~3.3]{FuKeller}, the map $X\mapsto C^T_{X}$ is a cluster character, with the property that $C^T_{T_i}=x_i$.

The cluster-tilting object $T$ also determines an index for each object $X\in\curly{E}$. To see that this quantity is well-defined we will use the following lemma, the proof of which is included for the convenience of the reader.

\begin{lemma}
\label{approximations-are-admissible}
Let $\curly{E}$ be an exact category, and let $M,T\in\curly{E}$.
\begin{itemize}
\item[(i)]If there exists an admissible epimorphism $T'\to M$ for $T'\in\add{T}$, then any right $\add{T}$-approximation of $M$ is an admissible epimorphism.
\item[(ii)]If there exists an admissible monomorphism $M\to T'$ for $T'\in\add{T}$, then any left $\add{T}$-approximation of $M$ is an admissible monomorphism.
\end{itemize}
\end{lemma}
\begin{proof}
We prove only (i), as (ii) is dual. Pick an admissible epimorphism $\pi\colon T'\to M$ with $T'\in\add{T}$ and a right $\add{T}$-approximation $f\colon R\to M$. Consider the pullback square
\[\begin{tikzcd}[column sep=20pt]
X\arrow{r}{g}\arrow{d}{\pi'}&T'\arrow{d}{\pi}\\
R\arrow{r}{f}&M
\end{tikzcd}\]
As $f$ is a right $\add{T}$-approximation, there is a map $h\colon T'\to R$ such that the square
\[\begin{tikzcd}[column sep=20pt]
T'\arrow{r}{1}\arrow{d}{h}&T'\arrow{d}{\pi}\\
R\arrow{r}{f}&M
\end{tikzcd}\]
commutes, and so by the universal property of pullbacks, there is $g'\colon T'\to X$ such that $gg'=1$. Thus $g$ is a split epimorphism, fitting into an exact sequence
\[\begin{tikzcd}[column sep=20pt]
0\arrow{r}&K\arrow{r}{i}&X\arrow{r}{g}&T'\arrow{r}&0.
\end{tikzcd}\]
It then follows, again by the universal property of pushouts, that $\pi'i$ is a kernel of $f$. Since $f\pi'=\pi g$ is the composition of two admissible epimorphisms, $f$ is itself an admissible epimorphism by the obscure axiom \cite[A.1]{KellerCC}, \cite[(Dual of) Proposition~2.16]{Buehler}.
\end{proof}

Given an object $X\in\curly{E}$, we may pick a minimal right $\add{T}$-approximation $R_X\to X$, where $R_X$ is determined up to isomorphism by $X$ and the existence of such a morphism. Let $P\to X$ be a projective cover of $X$, which exists since $\curly{E}$ has enough projectives; this is an admissible epimorphism by definition, and $P\in\add{T}$ since $T$ is cluster-tilting. Thus by Lemma~\ref{approximations-are-admissible}, the approximation $R_X\to X$ is an admissible epimorphism, and so there is an exact sequence
\[0\to K_X\to R_X\to X\to0\]
in $\curly{E}$. Since $T$ is cluster-tilting, $K_X\in\add{T}$, and we define $\ind{T}{X}=[R_X]-[K_X]\in\mathrm{K}_0(\add{T})$. It is crucial here that $\ind{T}{X}$ is defined in $\mathrm{K}_0(\add{T})$, rather than in $\mathrm{K}_0(\curly{E})$ where it would simply be equal to $[X]$.

We also associate to $T$ the exchange matrix $B_T$ given by the first $r$ columns of the antisymmetrisation of the incidence matrix of the quiver of $\Lambda$. By definition, $B_T$ has entries
\[(B_T)_{ij}=\dim\Ext{1}{\Lambda}{S_i}{S_j}-\dim\Ext{1}{\Lambda}{S_j}{S_i}\]
for $1\leq i\leq n$ and $1\leq j\leq r$.

\begin{definition}[{cf.\ \cite[Definition~3.3]{Pressland}}]\label{d:Frob-cl-cat}
A Frobenius category $\curly{E}$ is a Frobenius cluster category if it is Krull--Schmidt, stably $2$-Calabi--Yau and satisfies $\operatorname{gldim}(\op{\End{\curly{E}}{T}})\leq 3$ for all cluster-tilting objects $T\in\curly{E}$, of which there is at least one.
\end{definition}

Note that a Frobenius cluster category $\curly{E}$ need not be Hom-finite, but the stable category $\underline{\curly{E}}$ must be, since this is part of the definition of $2$-Calabi--Yau.

Let $\curly{E}$ be a Frobenius cluster category. Let $T=\bigdsum_{i=1}^n T_{i}\in \curly{E}$ be a basic cluster-tilting object, where each $T_i$ is indecomposable and is projective-injective if and only if $i>r$, let $\Lambda=\op{\End{\curly{E}}{T}}$ be its endomorphism algebra, and let $\underline{\Lambda}=\op{\End{\underline{\curly{E}}}{T}}$ be its stable endomorphism algebra. We continue to write $F=\Hom{\curly{E}}{T}{-}\colon \curly{E}\to\fgmod{\Lambda}$ and $E=\Ext{1}{\curly{E}}{T}{-}\colon\curly{E}\to\fgmod{\underline{\Lambda}}$. Since $\underline{\curly{E}}$ is Hom-finite, $\underline{\Lambda}$ is a finite dimensional algebra.

The Krull--Schmidt property for $\curly{E}$ is equivalent to $\curly{E}$ being idempotent complete and having the property that the endomorphism algebra $A$ of any of its objects is a semiperfect ring \cite[Corollary~4.4]{KrauseKS}, meaning there are a complete set $\{e_i:i\in I\}$ of pairwise orthogonal idempotents of $A$ such that $e_iAe_i$ is local for each $i\in I$. For many representation-theoretic purposes, semiperfect $\mathbb{K}$-algebras behave in much the same way as finite dimensional ones; for example, if $A$ is semiperfect then the quotient $A/\rad{A}$ is semi-simple, and its idempotents lift to $A$. For more background on semiperfect rings, see, for example, Anderson and Fuller \cite[Chapter~27]{AndersonFuller-Book}. 

For us, a key property of a semiperfect ring $A$ is that the $A$-modules $Ae_i/\rad{Ae_i}$ (respectively, their projective covers $Ae_i$) form a complete set of finite dimensional simple $A$-modules (respectively indecomposable projective $A$-modules) up to isomorphism \cite[Proposition~27.10]{AndersonFuller-Book}.  As we will require this later, we include being Krull--Schmidt in our definition of a Frobenius cluster category, noting that other work in this area---notably the original definition in \cite{Pressland}---requires only idempotent completeness. 

Since $\Lambda$ is Noetherian and $\operatorname{gldim}{\Lambda}\leq 3$, the Euler form
\[\ip{M}{N}_e=\sum_{i\geq0}(-1)^i\dim\Ext{i}{\Lambda}{M}{N}\]
is well-defined as a map $\mathrm{K}_0(\fgmod\Lambda)\times \mathrm{K}_0(\fd\Lambda)\to\integ$, and coincides with the form $\ip{-}{-}_3$ introduced earlier.

\begin{remark}
\label{exchangematrixgivesip}
By a result of Keller and Reiten \cite[\S4]{KellerReiten} (see also \cite[Theorem~3.4]{Pressland}), $\fgmod{\Lambda}$ has enough $3$-Calabi--Yau symmetry for us to deduce that $\dim\Ext{k}{\Lambda}{S_i}{S_j}=\dim\Ext{3-k}{\Lambda}{S_j}{S_i}$ when $1\leq j\leq r$. It follows that
\[(-B_T)_{ij}=\ip{S_i}{S_j}_3=\ip{S_i}{S_j},\]
so the matrix of $\ip{-}{-}$, when restricted to the span of the simple modules in the first entry and the span of the first $r$ simple modules in the second entry, is given by $-B_T$.
\end{remark}

One can show by taking projective resolutions that the classes $[P_i]$ of indecomposable projective $\Lambda$-modules span $\mathrm{K}_0(\fgmod{\Lambda})$. Moreover, since $\ip{P_i}{S_j}_e=\delta_{ij}$, any $x\in \mathrm{K}_0(\fgmod{\Lambda})$ has a unique expression
\[x=\sum_{i=1}^n\ip{x}{S_{i}}_{e}[P_i]\]
as a linear combination of the $[P_i]$, and so these classes in fact freely generate $\mathrm{K}_0(\fgmod{\Lambda})$.

Recall from the definition of the index that if $X\in\curly{E}$, there is an exact sequence
\[0\to K_X\to R_X\to X\to 0\]
in which $K_X$ and $R_X$ lie in $\add{T}$. Since $E$ vanishes on $\add{T}$, the functor $F$ takes the above sequence to a projective resolution
\[0\to FK_X\to FR_X\to FX\to 0\]
of $FX$ in $\fgmod{\Lambda}$. Thus $FX$ has projective dimension at most $1$, and so $\ip{FX}{-}_1=\ip{FX}{-}_e$. We can therefore rewrite the cluster character of $X$ as
\[C^T_{X}=\prod_{i=1}^nx_i^{\ip{FX}{S_i}_e}\sum_{v\in\integ^r}\chi(\text{Gr}_v(EX))\prod_{i=1}^nx_i^{-\ip{v}{S_i}_e}.\]

We now proceed to defining gradings for Frobenius cluster categories.  We can follow the same approach as in the triangulated case, using the index.  However, by \cite{FuKeller}, we have the following expansion of the index in terms of the classes of the indecomposable summands of $T$:
\[ \ind{T}{X}=\sum_{i=1}^n \ip{FX}{S_i}_{e}[T_{i}]\in \mathrm{K}_0(\add{T}). \]
Since $\Ext{1}{\Lambda}{T}{T}=0$, there are no non-split exact sequences in $\add{T}$, and so $\mathrm{K}_0(\add{T})$ is freely generated by the $[T_i]$. For the same reason, the functor $F$ is exact when restricted to $\add{T}$, and so induces a map $F_{\ast}\colon K_{0}(\add{T}) \to K_{0}(\fgmod{\Lambda})$, which takes $[T_i]$ to $[P_i]$, and so is an isomorphism. Applying this isomorphism to the above formula, we obtain $F_{\ast}(\ind{T}{X})=\sum \ip{FX}{S_{i}}_{e}[P_{i}]=[FX]$.

From this we see that if we wish to work concretely with matrix and vector entries, the index can be computed explicitly.  For the general theory, however, the equivalent K-theoretic expression is cleaner and so we shall phrase our definition of grading in those terms, the above observation showing us that this is equivalent to the approach in \cite{GradedCAs}.

We will define our $\abgp$-gradings to be certain elements of $\mathrm{K}_0(\fd{\Lambda})\otimes_\integ\abgp$. To state a suitable compatibility condition, it will be necessary to extend the Euler form to an $\integ$-bilinear form $\mathrm{K}_0(\fgmod{\Lambda})\times(\mathrm{K}_0(\fd{\Lambda})\otimes_\integ\abgp)\to\abgp$. In the by now familiar way, we do this using the $\integ$-module structure on $\abgp$, and, abusing notation, define
\[\ip{x}{\sum y_i\otimes a_i}_e=\sum \ip{x}{y_i}_ea_i.\]
It is straightforward to check that this form is well-defined and $\integ$-linear in each variable.

Thus we arrive at the following definition of a graded Frobenius cluster category, exactly analogous to Definitions~\ref{d:graded-gen-cl-cat} and \ref{d:degree-graded-gen-cl-cat} in the triangulated case.

\begin{definition} Let $\curly{E}$ be a Frobenius cluster category and $T$ a cluster-tilting object of $\curly{E}$ such that $\Lambda=\op{\End{\curly{E}}{T}}$ is Noetherian. We  say that $G\in\mathrm{K}_0(\fd{\Lambda})\otimes_\integ\abgp$ is a grading for $\curly{E}$ if $\ip{M}{G}_{e}=0$ for all $M\in\fgmod{\underline{\Lambda}}$.  We call $(\curly{E},T,G)$ a graded Frobenius cluster category.
\end{definition}

\begin{definition} Let $(\curly{E},T,G)$ be a graded Frobenius cluster category.  Define $\deg_{G}\colon \curly{E} \to \abgp$ by $\deg_{G}(X)=\ip{FX}{G}_{e}$.
\end{definition}

We record some straightforward consequences of the above definitions.

\begin{remark}\label{grading-remarks} {\ }
\begin{enumerate}[label=(\roman*)]
\item When considering $\integ$-gradings, we may use the natural isomorphism $\mathrm{K}_0(\fd\Lambda)\otimes_\integ\integ\stackrel{\sim}{\to}\mathrm{K}_0(\fd\Lambda)$ to think of a grading as an element of the Grothendieck group itself. Similarly, we can think of $\integ^m$-gradings as elements of $\mathrm{K}_0(\fd\Lambda)^m$.
\item Using the basis of simples for $\mathrm{K}_0(\fd{\Lambda})$, we can write $G=\sum_{i=1}^n[S_i]\otimes G_i$ for some unique $G_i\in\abgp$. Writing $\underline{G}\in\abgp^n$ for the column vector with entries $G_i$, the grading condition is equivalent to requiring $B_T^t\underline{G}=0$, by Remark~\ref{exchangematrixgivesip} and the assumption that $\underline{\Lambda}$ is finite dimensional.
\item Let $G_i$ be as in (ii). Since $FT_{i}=P_{i}$ and $\ip{P_i}{S_j}_{e}=\delta_{ij}$, we may compute
\[\deg_{G}(T_{i})=\ip{FT_{i}}{G}_{e}=G_i,\]
as expected.
\end{enumerate}
\end{remark}

The K-theoretic phrasing of the above definition leads us to the following observation.

\begin{lemma}\label{l:proj-inj-grading} Let $\curly{E}$ be Hom-finite, let $T\in\curly{E}$ be a cluster-tilting object with endomorphism algebra $\Lambda$ and let $V\in\curly{E}$ be projective-injective. Write $F=\Hom{\curly{E}}{T}{-}$. Then $[FV]\in \mathrm{K}_{0}(\fd \Lambda)$ is a $\integ$-grading for $\curly{E}$, and $\deg_{[FV]}(X)=\dim\Hom{\curly{E}}{X}{V}$.
\end{lemma}

\begin{proof} Letting $M\in \fgmod \underline{\Lambda}$, we need to check that $\ip{M}{FV}_{e}=0$.  By the internal Calabi--Yau property of $\fgmod \Lambda$ (see Remark~\ref{exchangematrixgivesip}), we may instead check that $\ip{FV}{M}_{e}=0$.  Firstly, $\Ext{i}{\Lambda}{FV}{M}=0$ for $i>0$ since $FV$ is projective.

Recall from above that there is an idempotent $e\in\Lambda$, given by projecting onto a maximal projective summand of $T$, such that $\underline{\Lambda}=\Lambda/\Lambda e\Lambda$.  Using this, $FV\in\add\Lambda e$ by the definition of $e$, and $\Hom{\Lambda}{\Lambda e}{M}=eM=0$ since $M$ is a $\underline{\Lambda}$-module.  Hence $\Hom{\Lambda}{FV}{M}=0$ also, so that $\ip{FV}{M}_{e}=\ip{M}{FV}_{e}=0$ as required.

By definition, $\deg_{[FV]}(X)=\dim\Hom{\Lambda}{FX}{FV}$ for $X\in\curly{E}$. Since $T$ is cluster-tilting, we have the short exact sequence
\[0\to K_X\to R_X\to X\to 0,\]
with $K_X,R_X\in\add{T}$, used to define the index. Applying $\Hom{\curly{E}}{-}{V}$, we obtain the exact sequence
\[0\to\Hom{\curly{E}}{X}{V}\to\Hom{\curly{E}}{R_X}{V}\to\Hom{\curly{E}}{K_X}{V}.\]
Alternatively, we can apply $\Hom{\Lambda}{F{-}}{FV}$ to obtain the exact sequence
\[0\to\Hom{\Lambda}{FX}{FV}\to\Hom{\Lambda}{FR_X}{FV}\to\Hom{\Lambda}{FK_X}{FV}.\]
Since $F$ restricts to an equivalence on $\add{T}$, and $V\in\add{T}$ since it is projective-injective, the right-hand maps in these two exact sequences are isomorphic, yielding an isomorphism $\Hom{\curly{E}}{X}{V}\cong\Hom{\Lambda}{FX}{FV}$ of their kernels, from which the result follows.
\end{proof}

This gives us a family of $\integ$-gradings canonically associated to any Hom-finite Frobenius cluster category; note that in fact we only need $FV=\Hom{\curly{E}}{T}{V}\in \fd \Lambda$, so for some specific Hom-infinite $\curly{E}$ and specific $V$ and $T$ the result may still hold.  

We will give some more examples of gradings later but first give the main results regarding graded Frobenius cluster categories, analogous to those in Proposition~\ref{p:prop-of-gen-cc} for the triangulated case.  We treat the straightforward parts first.

\begin{proposition} Let $(\curly{E},T,G)$ be a graded Frobenius cluster category.
\begin{enumerate}[label=(\roman*)]
\item Let $\mathbb{C}[x_{1}^{\pm 1},\dotsc ,x_{n}^{\pm 1}]$ be graded by $\deg_{G}(x_{j})=G_i$, where $G_i$ is defined as in Remark~\ref{grading-remarks}(ii). Then for all $X \in \curly{E}$, the cluster character $C_{X}^{T}\in \mathbb{C}[x_{1}^{\pm 1},\dotsc ,x_{n}^{\pm 1}]$ is homogeneous of degree $\deg_G(X)$.
\item\label{p:prop-of-gen-cc-additive-on-exact-seq} For any exact sequence $0\to X\to Y \to Z \to 0$ in $\curly{E}$, we have
\[ \deg_{G}(Y)=\deg_{G}(X)+\deg_{G}(Z). \]
\item The degree $\deg_{G}$ is compatible with mutation in the sense that for every cluster-tilting object $U$ of $\curly{E}$ with indecomposable summand $U_{k}$ we have
\[ \deg_{G}(U_{k}^{*})=\deg_{G}(M)-\deg_{G}(U_{k})=\deg_{G}(M')-\deg_{G}(U_{k}), \]
where $U_{k}^{*}$, $M$ and $M'$ are as in the above description of exchange sequences in $\curly{E}$. It follows that $\deg_G(M)=\deg_G(M')$, which is the categorical version of the claim that all exchange relations in a graded cluster algebra are homogeneous.
\end{enumerate}
\end{proposition}
\begin{proof} {\ }
\begin{enumerate}[label=(\roman*)]
\item As usual, for $v\in\integ^n$ we write $\underline{x}^v=\prod_{i=1}^nx_i^{v_i}$. Then if $\deg_{G}(x_i)=G_i$, we have
\[\deg_{G}(\underline{x}^v)=\sum_{i=1}^nv_iG_i=\ip{\sum_{i=1}^nv_i[P_i]}{G}_{e}.\]
Each term of $C_X^T$ may be written in the form form $\lambda \underline{x}^v$, where
\[v_i=\ip{FX}{S_i}_{e}-\ip{M}{S_i}_{e}\]
for some $M\in\fgmod{\underline{\Lambda}}$, and $\lambda$ is a constant. It follows that
\[\sum_{i=1}^nv_i[P_i]=[FX]-[M],\]
so the degree of $\underline{x}^v$ is
\[\ip{FX}{G}_{e}-\ip{M}{G}_{e}=\ip{FX}{G}_{e}=\deg_{G}(X),\]
since $\ip{M}{G}_{e}=0$ by the definition of a grading. In particular, this is independent of $M$, so $C^T_X$ is homogeneous of degree $\deg_{G}(X)$.

\item Applying $F$ to the exact sequence $0\to X\to Y\to Z\to 0$ and truncating gives an exact sequence
\[0\to FX\to FY\to FZ\to M\to0\]
for some $M\subseteq EX$. In particular, $M\in\fgmod{\underline{\Lambda}}$. In $\mathrm{K}_0(\fgmod{\Lambda})$, we have
\[[FX]+[FZ]=[FY]+[M],\]
so applying $\ip{-}{G}_e$ gives
\[\deg_G(X)+\deg_G(Z)=\deg_G(Y)+\ip{M}{G}_e=\deg_G(Y)\]
since $M\in\fgmod{\underline{\Lambda}}$.

\item This follows directly from (ii) applied to the exchange sequences
\[0\to U_{k}^{*} \to M \to U_{k} \to 0 \qquad \text{and} \qquad 0\to U_{k} \to M' \to U_{k}^{*} \to 0.\qedhere\]
\end{enumerate}
\end{proof}

Since the shift functor on $\underline{\curly{E}}$ does not typically lift to an automorphism of $\curly{E}$, and projective-injective objects of $\curly{E}$ may have non-zero degrees, we have no natural analogue of Proposition~\ref{p:prop-of-gen-cc}(v) in the Frobenius setting. It remains to give an analogue of part~\ref{p:prop-of-gen-cc-Groth-gp}, concerning the relationship between gradings and the Grothendieck group of a graded Frobenius cluster category.  The first part of the following theorem is directly analogous to \cite[Theorem 10]{Palu-Groth-gp} for the triangulated case.

\begin{theorem}\label{t:grading-Groth-gp}
  Let $\curly{E}$ be a Frobenius cluster category with a cluster-tilting object $T$ such that $\Lambda=\op{\End{\curly{E}}{T}}$ is Noetherian.
  \begin{enumerate}[label=(\roman*)]
    \item\label{t:grading-Groth-gp-relns} The Grothendieck group $\mathrm{K}_{0}(\curly{E})$, as an exact category, is isomorphic to the quotient of $\mathrm{K}_{0}(\add_{\curly{E}} T)$ by the relations $[X_{k}]-[Y_{k}]$, for $1\leq k\leq r$, where \[0\to U_{k}^{*} \to Y_k \to U_{k} \to 0 \qquad \text{and} \qquad 0\to U_{k} \to X_k \to U_{k}^{*} \to 0\]
      are the exchange sequences associated to the summand $U_k$ of $T$.
    \item\label{t:grading-Groth-gp-grading-space} The space of $\abgp$-gradings of $\curly{E}$, defined above as a subspace of $\mathrm{K}_0(\fd{\Lambda})\otimes_\integ\abgp$, is isomorphic to $\Hom{\integ}{\mathrm{K}_0(\curly{E})}{\abgp}$, via the map $G\mapsto\deg_G$.
      \end{enumerate}
\end{theorem}

\begin{proof}
Let $\curly{H}^b(\add_{\curly{E}}{T})$ denote the bounded homotopy category of complexes with terms in $\add_{\curly{E}}{T}$, and let $\curly{H}^b_{\curly{E}\text{-ac}}(\add_{\curly{E}}{T})$ denote the full subcategory of $\curly{E}$-acyclic complexes. By work of Palu \cite[Lemma~2]{Palu-Groth-gp}, there is an exact sequence
\[\begin{tikzcd}[column sep=20pt]
0\arrow{r}&\curly{H}^b_{\curly{E}\text{-ac}}(\add_{\curly{E}}{T})\arrow{r}&\curly{H}^b(\add_{\curly{E}}{T})\arrow{r}&\curly{D}^b(\curly{E})\arrow{r}&0,
\end{tikzcd}\]
of triangulated categories, to which we may apply the right exact functor $\mathrm{K}_0$ to obtain
\[\begin{tikzcd}[column sep=20pt]
\mathrm{K}_0(\curly{H}^b_{\curly{E}\text{-ac}}(\add_{\curly{E}}{T}))\arrow{r}&\mathrm{K}_0(\curly{H}^b(\add_{\curly{E}}{T}))\arrow{r}&\mathrm{K}_0(\curly{D}^b(\curly{E}))\arrow{r}&0.
\end{tikzcd}\]

By \cite[Proof of Lemma~9]{Palu-Groth-gp}, there is a natural isomorphism $\mathrm{K}_0(\curly{H}^b_{\curly{E}\text{-ac}}(\add_{\curly{E}}{T}))\stackrel{\sim}{\to}\mathrm{K}_0(\fgmod{\underline{\Lambda}})$. Moreover, since $T$ is cluster-tilting, there are no non-split exact sequences in $\add{T}$, and so $\mathrm{K}_0(\add{T})$ is freely generated by the indecomposable summands of $T$. Thus taking the alternating sum of terms gives an isomorphism $\mathrm{K}_0(\curly{H}^b(\add_{\curly{E}}{T}))\stackrel{\sim}{\to}\mathrm{K}_0(\add_{\curly{E}}{T})$ \cite{Rose-Note}.

These isomorphisms induce a commutative diagram
\[\begin{tikzcd}[column sep=20pt]
\mathrm{K}_0(\curly{H}^b_{\curly{E}\text{-ac}}(\add_{\curly{E}}{T}))\arrow{r}\arrow{d}&\mathrm{K}_0(\curly{H}^b(\add_{\curly{E}}{T}))\arrow{r}\arrow{d}&\mathrm{K}_0(\curly{D}^b(\curly{E}))\arrow{r}\arrow{d}&0\\
\mathrm{K}_0(\fgmod{\underline{\Lambda}})\arrow{r}{\varphi}&\mathrm{K}_0(\add_{\curly{E}}{T})\arrow{r}&\mathrm{K}_0(\curly{E})\arrow{r}&0
\end{tikzcd}\]
with exact rows. Since the two leftmost vertical maps are isomorphisms, the induced map $\mathrm{K}_0(\curly{D}^b(\curly{E}))\to\mathrm{K}_0(\curly{E})$, which is again given by taking the alternating sum of terms, is also an isomorphism.

We claim that the map $\varphi$ in the above diagram is given by composing the map from $\mathrm{K}_0(\fgmod{\underline{\Lambda}})$ to $\mathrm{K}_0(\fgmod{\Lambda})$ induced by the inclusion of categories with the inverse of the isomorphism $F_*\colon\mathrm{K}_0(\add_{\curly{E}}{T})\stackrel{\sim}{\to}\mathrm{K}_0(\fgmod{\Lambda})$. Since $\underline{\Lambda}$ is finite dimensional, the Grothendieck group $\mathrm{K}_0(\fgmod{\underline{\Lambda}})$ is spanned by the classes of the simple $\underline{\Lambda}$-modules $S_k$ for $1\leq k\leq r$, so it suffices to check that $\varphi$ acts on these classes as claimed. Let
\[0\to U_{k}^{*} \to Y_k \to U_{k} \to 0 \qquad \text{and} \qquad 0\to U_{k} \to X_k \to U_{k}^{*} \to 0\]
be the exchange sequences associated to the summand $U_k$ of $T$. Then there is an exact sequence
\[0\to FU_k\to FX_k\to FY_k\to FU_k\to S_k\to0.\]
From this we see that $[S_k]=[FX_k]-[FY_k]=F_*([X_k]-[Y_k])$ in $\mathrm{K}_0(\fgmod{\Lambda})$, and so we want to show that $\varphi[S_k]=[X_k]-[Y_k]$. On the other hand, $[S_k]$ is the image of the class of the $\curly{E}$-acyclic complex
\[\cdots\to0\to U_k\to X_k\to Y_k\to U_k\to0\to\cdots\]
under Palu's isomorphism $\mathrm{K}_0(\curly{H}^b_{\curly{E}\text{-ac}}(\add_{\curly{E}}{T}))\stackrel{\sim}{\to}\mathrm{K}_0(\fgmod{\underline{\Lambda}})$ (\cf \cite[Proof of Theorem~10]{Palu-Groth-gp}), and the image $\varphi[S_k]$ of this complex in $\mathrm{K}_0(\add_{\curly{E}}{T})$ is $[X_k]-[Y_k]$, as we wanted.  This yields \ref{t:grading-Groth-gp-relns}.

Now applying $\Hom{\integ}{-}{\abgp}$ to the exact sequence
\[\begin{tikzcd}[column sep=20pt]
\mathrm{K}_0(\fgmod{\underline{\Lambda}})\arrow{r}{\varphi}&\mathrm{K}_0(\add_{\curly{E}}{T})\arrow{r}&\mathrm{K}_0(\curly{E})\arrow{r}&0
\end{tikzcd}\]
shows that $\Hom{\integ}{\mathrm{K}_0(\curly{E})}{\abgp}$ is isomorphic to the kernel of $\varphi^t=\Hom{\integ}{\varphi}{\abgp}$, which we will show coincides with the space of gradings. Indeed, we may identify $\mathrm{K}_0(\add_{\curly{E}}{T})$ with $\mathrm{K}_0(\fgmod{\Lambda})$ via $F_*$, and then use the Euler form to identify $\mathrm{K}_0(\fd{\Lambda})\otimes_\integ\abgp$ with $\Hom{\integ}{\mathrm{K}_0(\fgmod{\Lambda})}{\abgp}$, the map
\[x\mapsto\ip{-}{x}_e\]
being an isomorphism as usual. Under this identification, we have $\varphi^tG=\ip{-}{G}_e|_{\mathrm{K}_0(\fgmod{\underline{\Lambda}})}$, and so $G\in\ker{\varphi}^t$ if and only if it is a grading.
The claim that the isomorphism is given explicitly by $G\mapsto\deg_G=\ip{F(-)}{G}_e$ can be seen by diagram chasing, and hence \ref{t:grading-Groth-gp-grading-space} is proved.
\end{proof}

The significance of this theorem is that, as in the triangulated case, it provides a basis-free method to identify gradings on Frobenius cluster categories and the cluster algebras they categorify.  In the latter context, basis-free essentially means free of the choice of a particular cluster.

Specifically, as explained in more detail below, to establish that some categorical datum gives a grading, one only needs to check that that it respects exact sequences.  This is potentially significantly easier than checking the vanishing of the product $B_{T}^t\underline{G}$ where $B_{T}$ is given in terms of dimensions of $\text{Ext}$-spaces over the endomorphism algebra $\Lambda$ of some cluster-tilting object $T$.

On the other hand, given some knowledge of the cluster algebra being categorified---in particular, knowing a seed---one can use the above theorem to deduce information about the Grothendieck group of the Frobenius cluster category.

As promised in Section~\ref{preliminaries}, we can use Theorem~\ref{t:grading-Groth-gp} to see how the grading in a graded Frobenius cluster category is independent of the cluster-tilting object. Precisely, let $(\curly{E},T,G)$ be a graded Frobenius cluster category, and let $\deg_G$ be the corresponding function on $\mathrm{K}_0(\curly{E})$. Let $T'=\bigoplus_{i=1}^nT_i'$ be another cluster-tilting object, with $\Lambda'=\op{\End{\curly{E}}{T'}}$, and denote the simple $\Lambda'$-modules by $S_i'$ for $1\leq i\leq n$. Using the inverse of the isomorphism of Theorem~\ref{t:grading-Groth-gp}, we see that if $G'$ in $\mathrm{K}_0(\fd\Lambda')$ is given by
\[G'=\sum_{i=1}^n\deg_G(T_i')[S_i'],\]
then $(\curly{E},T',G')$ is a graded Frobenius cluster category with $\deg_G=\deg_{G'}$, as one should expect. Note that this statement holds even if, as can happen, there is no sequence of mutations from $T$ to $T'$.

As was remarked about the triangulated case in \cite{GradedCAs}, these observations highlight how the categorification of a cluster algebra is able to see global properties, whereas the algebraic combinatorial mutation process is local.

The following example shows the theorem in action, although again we need the additional assumption of Hom-finiteness of $\curly{E}$.

\begin{lemma}\label{l:dim-vector} Assume that $\curly{E}$ is Hom-finite and let $P$ be a projective-injective object.  Then $\dim \Hom{\curly{E}}{P}{-}$ and $\dim\Hom{\curly{E}}{-}{P}$ define $\integ$-gradings for $\curly{E}$.
\end{lemma}

\begin{proof} Since $P$ is projective and injective, both $\Hom{\curly{E}}{P}{-}$ and $\Hom{\curly{E}}{-}{P}$ are exact functors, and so in each case taking the dimension yields a function in $\Hom{\integ}{\mathrm{K}_0(\curly{E})}{\integ}$. Then the result follows immediately from Theorem~\ref{t:grading-Groth-gp}.
\end{proof}

In sufficiently nice cases, applying this result with a complete set of indecomposable projectives will yield that the dimension vector of a module is a (multi-)grading.  

However, we remark that some care may be needed regarding which algebra we measure ``dimension vector'' over.  If $\curly{E}\subset\fgmod{\Pi}$ for some algebra $\Pi$ (as in most examples), then we may consider the $\Pi$-dimension vector of $X\in\curly{E}$, defined in the usual way. On the other hand, any Hom-finite Frobenius cluster category $\curly{E}$ is equivalent to $\GP(B)\subset\fgmod{B}$ for $B$ the opposite endomorphism algebra of a basic projective generator $P=\bigoplus_{i=1}^nP_i$ of $\curly{E}$, by \cite[Theorem~2.7]{KIWY}.  Re-interpreting all of the objects of $\curly{E}$ as $B$-modules, the projective-injectives will now be precisely the projective $B$-modules, and $(\dim\Hom{\curly{E}}{P_i}{X})$ is the $B$-dimension vector of $X$ (tautologically, since the equivalence $\curly{E}\to\GP(B)$ takes $X$ to $\Hom{\curly{E}}{P}{X}$). Note that $B$ may not be the same as the algebra $\Pi$ from which $\curly{E}$ originated, and the $B$-dimension vector of a module may differ from the $\Pi$-dimension vector.

Given a complete set of projectives, it is natural to ask whether the associated grading might be standard, as defined in \cite{GradedCAs}; we briefly recall this definition and some related facts.

\begin{definition}  Let $(\underline{x},B)$ be a seed.  We call a multi-grading $G$ whose columns are a basis for the kernel of $B$ a standard multi-grading, and call $(\underline{x},B,G)$ a standard graded seed.
\end{definition}

It is straightforward to see, from rank considerations, that mutation preserves the property of being standard.  Moreover, as shown in \cite{GradedCAs}, if $(\underline{x},B,G)$ is a standard graded seed and $H$ is any grading for $(\underline{x},B)$, then there exists an integer matrix $M=M(G,H)$ such that for any cluster variable $y$ in $\curly{A}(\underline{x},B,H)$ we have 
\[ \deg_{H}(y)=\deg_{G}(y)M, \]
where on the right-hand side we regard $y$ as a cluster variable of $\curly{A}(\underline{x},B,G)$ in the obvious way.

That is, to describe the degree of a cluster variable of a graded cluster algebra $\curly{A}(\underline{x},B,H)$, it suffices to know its degree with respect to some standard grading $G$ and the matrix $M=M(G,H)$ transforming $G$ to $H$.  In particular, to understand the distribution of the degrees of cluster variables, it suffices to know this for standard gradings.

Since the statement applies in the particular case when $G$ and $H$ are both standard, we see that from one choice of basis for the kernel of $B$, we obtain complete information.  For if we chose a second basis, the change of basis matrix tells us how to transform the degrees.  Hence up to a change of basis, there is essentially only one standard grading for each seed.  

Then, depending on the particular Frobenius cluster category at hand, if we have knowledge of the rank of the exchange matrix, we may be able to examine categorical data such as the number of projective-injective modules or dimension vectors and hence try to find a basis for the space of gradings.  

For example, for a basic cluster-tilting object $T$ in $\curly{E}$ a Hom-finite Frobenius cluster category, we have $n-r$ projective-injective summands in $T$: if the exchange matrix $B_{T}$ has full rank, a basis for the space of gradings has size $n-r$ so that, via Lemma~\ref{l:proj-inj-grading}, a canonical standard grading is given by the set $\{ [FT_{i}] \mid i>r \}$, which is linearly independent since it is a subset of the basis of projectives for $\mathrm{K}_0(\fd\Lambda)=\mathrm{K}_0(\fgmod{\Lambda})$.

From knowledge of this standard grading, we then obtain any other grading by means of some linear transformation.  In the next section, we do this for two important examples.

\section{Examples of graded Frobenius cluster categories}

\subsection{Frobenius cluster categories associated to partial flag varieties}

Let $\mathfrak{g}$ be the Kac--Moody algebra associated to a symmetric generalised Cartan matrix.  Let $\Delta$ be the associated Dynkin graph and pick an orientation $\vec{\Delta}$. Let $Q$ be the quiver obtained from $\vec{\Delta}$ by adding an arrow $\alpha^*\colon j\to i$ for each arrow $\alpha\colon i\to j$ of $\vec{\Delta}$. Then the preprojective algebra of $\Delta$ is
\[\Pi=\complex Q/\sum_{\alpha\in\vec{\Delta}}[\alpha,\alpha^*],\]
which is, up to isomorphism, independent of the choice of orientation $\vec{\Delta}$.

For each $w\in W$, the Weyl group of $\mathfrak{g}$, Buan, Iyama, Reiten and Scott \cite{BIRS1} have introduced a category $\curly{C}_{w}$; the following version of its construction follows \cite{GLS-KacMoody}, and is dual to the original.

Assume $w$ has finite length and set $l(w)=n$; we do this for consistency with the notation used above but note that other authors (notably \cite{GLS-KacMoody}, \cite{GLS-QuantumPFV}) use $r$ and their $n$ is our $n-r$.

Set $\hat{I}_{i}$ to be the indecomposable injective $\Pi$-module with socle $S_{i}$, the 1-dimensional simple module supported at the vertex $i$ of $Q$.  

Given a module $W$ in $\fgmod \Pi$, we define
\begin{itemize}
\item $\mathrm{soc}_{(l)}(W)\defeq {\displaystyle \sum_{\substack{U\leq W \\ U\iso S_{l}}} U}$ and
\item $\mathrm{soc}_{(l_{1},l_{2},\ldots,l_{s})}(W)\defeq W_{s}$ where the chain of submodules $0=W_0\subseteq W_{1} \subseteq \cdots \subseteq W_{s} \subseteq W$ is such that $W_{p}/W_{p-1} \iso \mathrm{soc}_{(l_{p})}(W/W_{p-1})$.
\end{itemize}
Let $\mathbf{i}=(i_n,\dotsc,i_1)$ be a reduced expression for $w$. Then for $1\leq k \leq n$, we define $V_{\mathbf{i},s} \defeq \mathrm{soc}_{(i_{k},i_{s-1},\ldots,i_{1})}(\hat{I}_{i_{s}})$.  Set $V_{\mathbf{i}}=\bigdsum_{k=1}^{n} V_{\mathbf{i},k}$ and let $I$ be the subset of $\{ 1,\dotsc ,n\}$ such that the modules $V_{\mathbf{i},i}$ for $i\in I$ are $\curly{C}_{w}$-projective-injective.  Set $I_{\mathbf{i}}=\bigdsum_{i\in I} V_{\mathbf{i},i}$ and $n-r=\card{I}$.  Note that this is also the number of distinct simple reflections appearing in $\mathbf{i}$.

Define 
\[ \curly{C}_{\mathbf{i}}=\operatorname{Fac}(V_{\mathbf{i}})\subseteq \text{nil}\ \Pi. \]
That is, $\curly{C}_{\mathbf{i}}$ is the full subcategory of $\fgmod \Pi$ consisting of quotient modules of direct sums of finitely many copies of $V_{\mathbf{i}}$.

Then $\curly{C}_{\mathbf{i}}$ and $I_{\mathbf{i}}$ are independent of the choice of reduced expression $\mathbf{i}$ (although $V_{\mathbf{i}}$ is not), so that we may write $\curly{C}_{w}\defeq\curly{C}_{\mathbf{i}}$ and $I_w\defeq I_{\mathbf{i}}$. It is shown in \cite{BIRS1} that $\curly{C}_{w}$ is a stably 2-Calabi--Yau Frobenius category.  Moreover $\curly{C}_{w}$ has cluster-tilting objects: $V_{\mathbf{i}}$ is one such.  Indeed, cluster-tilting objects are maximal rigid, and vice versa.  The indecomposable $\curly{C}_{w}$-projective-injective modules are precisely the indecomposable summands of $I_{w}$, and $\curly{C}_{w}=\text{Fac}(I_{w})$.

Furthermore, it is also shown in \cite[Proposition~2.19]{GLS-KacMoody} that the global dimension condition of Definition~\ref{d:Frob-cl-cat} also holds, leaving only the Krull--Schmidt condition. By \cite[Corollary~4.4]{KrauseKS}, we should check that the endomorphism algebras of objects of $\curly{C}_w$ are semiperfect, and that this category is idempotent complete. The first of these properties holds since $\curly{C}_w$ is Hom-finite. The second follows from the fact that $\curly{C}_w$ is a full subcategory of the idempotent complete category $\fgmod(\Pi/\operatorname{Ann}{I_w})$, and that if $M$ is an object of $\operatorname{Fac}(I_w)$, then so are all direct summands of $M$.

We conclude that $\curly{C}_{w}$ is a Frobenius cluster category, in the sense of Definition~\ref{d:Frob-cl-cat}.

Let $\Lambda=\op{\End{\curly{C}_{w}}{V_{\mathbf{i}}}}$ and $F=\Hom{\curly{C}_{w}}{V_{\mathbf{i}}}{-}$.  Then, as above, the modules $P_{k}\defeq FV_{\mathbf{i},k}$ for $1\leq k\leq n$ are the indecomposable projective $\Lambda$-modules and the tops of these, $S_{k}$, are the simple $\Lambda$-modules.  Recall that the exchange matrix obtained from the quiver of $\Lambda$, which we shall call $B_{\mathbf{i}}$, has entries 
\[(B_{\mathbf{i}})_{ij}=\dim\Ext{1}{\Lambda}{S_i}{S_j}-\dim\Ext{1}{\Lambda}{S_j}{S_i}\]
for $1\leq i\leq n$ and $j\notin I$, so that the $r$ columns of $B_{\mathbf{i}}$ correspond to to the mutable summands $V_{\mathbf{i},j}$, $j\notin I$, of $V_{\mathbf{i}}$.

Let $L_{\mathbf{i}}$ be the $n\cross n$ matrix with entries 
\[ (L_{\mathbf{i}})_{jk}=\dim\Hom{\Pi}{V_{\mathbf{i},j}}{V_{\mathbf{i},k}}-\dim\Hom{\Pi}{V_{\mathbf{i},k}}{V_{\mathbf{i},j}}. \]

\noindent By \cite[Proposition~10.1]{GLS-QuantumPFV} we have 
\[ \sum_{l=1}^{n} (B_{\mathbf{i}})_{lk}(L_{\mathbf{i}})_{lj}=2\delta_{jk}, \]
and hence the matrix $B_{\mathbf{i}}$ has maximal rank, namely $r$.  

It follows that there exists some standard integer multi-grading $G_{\mathbf{i}}=(G_{1},\dotsc ,G_{n-r})\in \mathrm{K}_{0}(\fgmod{\Lambda})^{n-r}$ for $\curly{C}_{w}$ and $(\curly{C}_{w},V_{\mathbf{i}},G_{\mathbf{i}})$ is a graded Frobenius cluster category. As discussed above, such a standard grading can be used to construct all other gradings, so our goal is to identify one.

We have additional structure on $\curly{C}_{w}$ that we may make use of.  Namely, $\curly{C}_{w}$ is Hom-finite and we may apply Lemma~\ref{l:proj-inj-grading} with respect to the $\curly{C}_{w}$-projective-injective modules $V_{\mathbf{i},i}$ that are the indecomposable summands of $I_{\mathbf{i}}$.

The resulting grading $[FV_{\mathbf{i},i}]$, $i\in I$, is standard, since its $n-r$ components are a subset of the basis of projectives for $\mathrm{K}_0(\fgmod{\Lambda})$, and so in particular are linearly independent. By Theorem~\ref{t:grading-Groth-gp}, the existence of this standard grading implies that the Grothendieck group $\mathrm{K}_0(\curly{C}_w)$ has rank $n-r$.

We wish to understand this standard grading more explicitly.  Note that the objects of $\curly{C}_{w}$ are $\Pi$-modules and we may consider dimension vectors with respect to the $\Pi$-projective modules.

Then we notice that in fact the grading by $([FV_{\mathbf{i},i}])_{i\in I}$ is equal to the $\Pi$-dimension vector grading in the case at hand.  This is because, by Lemma~\ref{l:proj-inj-grading}, the degree of $X$ with respect to $[FV_{\mathbf{i},i}]$ is $\dim\Hom{\Pi}{X}{V_{\mathbf{i},i}}$, and each $V_{\mathbf{i},i}$ is both a submodule and a minimal right $\curly{C}_w$-approximation of an indecomposable injective $\hat{I}_{i}$ for $\Pi$, so $\Hom{\Pi}{X}{V_{\mathbf{i},i}}=\Hom{\Pi}{X}{\hat{I}_{i}}$, the dimensions of the latter giving the $\Pi$-dimension vector of $X$.

In \cite[Corollary~9.2]{GLS-KacMoody}, Gei\ss, Leclerc and Schr\"{o}er have shown that
\[ \dimvec_{\Pi} V_{\mathbf{i},k}=\omega_{i_{k}}-s_{i_{1}}s_{i_{2}}\dotsm s_{i_{k}}(\omega_{i_{k}})\]
for all $1\leq k\leq n$, where the $\omega_{j}$ are the fundamental weights for $\mathfrak{g}$ and the $s_{j}$ the Coxeter generators for $W$.  This enables us to construct the above grading purely combinatorially.

\begin{example}
We consider the following seed associated to $\mathfrak{g}$ of type $A_{5}$ with\[ \mathbf{i}=(3,2,1,4,3,2,5,4,3), \] as given in \cite[Example~12.11]{GLS-QuantumPFV}.  The modules $V_{k}\defeq V_{\mathbf{i},k}$, in terms of the usual representation illustrating their composition factors as $\Pi$-modules, are 


\begin{align*}
V_1&=\begin{smallmatrix}3\end{smallmatrix}&
V_2&=\begin{smallmatrix}3\\&4\end{smallmatrix}&
V_3&=\begin{smallmatrix}3\\&4\\&&5\end{smallmatrix}\\\\
V_4&=\begin{smallmatrix}&3\\2\end{smallmatrix}&
V_5&=\begin{smallmatrix}&3\\2&&4\\&3\end{smallmatrix}&
V_6&=\begin{smallmatrix}&3\\2&&4\\&3&&5\\&&4\end{smallmatrix}\\\\
V_7&=\begin{smallmatrix}&&3\\&2\\1\end{smallmatrix}&
V_8&=\begin{smallmatrix}&&3\\&2&&4\\1&&3\\&2\end{smallmatrix}&
V_9&=\begin{smallmatrix}&&3\\&2&&4\\1&&3&&5\\&2&&4\\&&3\end{smallmatrix}
\end{align*}

The exchange quiver for this seed is

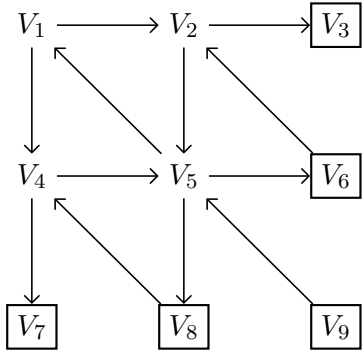
\begin{center}
\scalebox{1}{\begin{tikzpicture}[node distance=2cm,on grid,>=angle 90]

\node (11) at (0,0) {$V_{1}$}; 
\node (12) [right=of 11] {$V_{2}$};
\node (13) [right=of 12,rectangle,draw=black,thick] {$V_{3}$};

\node (21) [below=of 11] {$V_{4}$}; 
\node (22) [right=of 21] {$V_{5}$};
\node (23) [right=of 22,rectangle,draw=black,thick] {$V_{6}$};

\node (31) [below=of 21,rectangle,draw=black,thick] {$V_{7}$}; 
\node (32) [right=of 31,rectangle,draw=black,thick] {$V_{8}$};
\node (33) [right=of 32,rectangle,draw=black,thick] {$V_{9}$};

\draw[semithick,->] (11) to (12);
\draw[semithick,->] (12) to (13);

\draw[semithick,->] (21) to (22);
\draw[semithick,->] (22) to (23);

\draw[semithick,->] (11) to (21);
\draw[semithick,->] (21) to (31);

\draw[semithick,->] (12) to (22);
\draw[semithick,->] (22) to (32);

\draw[semithick,->] (22) to (11);
\draw[semithick,->] (23) to (12);

\draw[semithick,->] (32) to (21);
\draw[semithick,->] (33) to (22);

\end{tikzpicture}}
\end{center}

It is straightforward to see that $\Pi$-dimension vectors yield a grading: for example, looking at the vertex corresponding to $V_{1}$, the sums of the dimension vectors of incoming and outgoing arrows are $[0,1,2,1,0]$ and $[0,1,1,0,0]+[0,0,1,1,0]$ respectively.
\end{example}
 
\subsection{Grassmannian cluster categories}

Let $\Pi$ be the preprojective algebra of type $\mathsf{A}_{n-1}$, with vertices numbered sequentially, and let $Q_k$ be the injective module at the $k$th vertex. In \cite{GLS-PFV}, Gei\ss, Leclerc and Schr\"oer show that the category $\operatorname{Sub}\, Q_{k}$ of submodules of direct sums of copies of $Q_k$ ``almost'' categorifies the cluster algebra structure on the homogeneous coordinate ring of the Grassmannian of $k$-planes in $\complex^n$, but is missing a single indecomposable projective object corresponding to one of the frozen variables of this cluster algebra. The category $\Sub{Q_k}$ is in fact dual to one of the categories $\curly{C}_w$ introduced in the previous section, for $\Delta=\mathsf{A}_{n-1}$ and $w$ a particular Weyl group element depending on $k$, so it is a Frobenius cluster category in the same way.

Jensen, King and Su \cite{JKS} complete the categorification via the category $\CM(A)$ of maximal Cohen--Macaulay modules for a Gorenstein order $A$ (depending on $k$ and $n$) over $Z=\powser{\mathbb{C}}{t}$. One description of $A$ is as follows. Let $\Delta$ be the graph (of affine type $\tilde{\mathsf{A}}_{n-1}$) with vertex set given by the cyclic group $\integ_n$, and edges between vertices $i$ and $i+1$ for all $i$. Let $\Pi$ be the completion of the preprojective algebra on $\Delta$ with respect to the arrow ideal. Write $x$ for the sum of ``clockwise'' arrows $i\to i+1$, and $y$ for the sum of ``anti-clockwise'' arrows $i\to i-1$. Then we have
\[A=\Pi/\langle x^k-y^{n-k}\rangle.\]
In this description, $Z$ may be identified with the centre $\powser{\mathbb{C}}{xy}$ of $A$.

Jensen, King and Su also show \cite[Theorem~4.5]{JKS} that there is an exact functor \linebreak $\pi\colon \CM(A) \to \operatorname{Sub}\, Q_{k}$, corresponding to the quotient by the ideal generated by $P_{n}$, and that for any $N\in \operatorname{Sub}\, Q_{k}$, there is a unique (up to isomorphism) minimal $M$ in $\CM(A)$ with $\pi M\iso N$ and $M$ having no summand isomorphic to $P_{n}$.  Such an $M$ satisfies $\mathrm{rk}(M)=\dim \mathrm{soc}\ \pi M$, where $\mathrm{rk}(M)$ is the rank of each vertex component of $M$, thought of as a $Z$-module.

We now show that $\CM(A)$ is again a Frobenius cluster category. Properties of the algebra $A$ mean that an $A$-module is maximal Cohen--Macaulay if and only if it is  free and finitely generated as a $Z$-module. Since $Z$ is a principal ideal domain, and hence Noetherian, any submodule of a free and finitely generated $Z$-module is also free and finitely generated, and so $\CM(A)$ is closed under subobjects. In particular, $\CM(A)$ is closed under kernels of epimorphisms. Moreover \cite[Corollary~3.7]{JKS}, $A\in\CM(A)$, and so $\Omega(\fgmod{A})\subseteq\CM(A)$.

As a $Z$-module, any object $M\in\CM(A)$ is isomorphic to $Z^k$ for some $k$, so we have that $\op{\End{Z}{M}}\cong Z^{k^2}$ is a finitely generated $Z$-module. Since $Z$ is Noetherian, the algebra \linebreak $\op{\End{A}{M}}\subseteq\op{\End{Z}{M}}$ is also finitely generated as a $Z$-module. Thus $\op{\End{A}{M}}$ is Noetherian, as it is finitely generated as a module over the commutative Noetherian ring $Z$. We may now apply \cite[Proposition~3.6]{Pressland} to see that any cluster-tilting object $T\in\CM(A)$ satisfies $\operatorname{gldim}{\op{\End{A}{T}}}\leq 3$. Moreover \cite[Corollary~4.6]{JKS}, $\underline{\CM}(A)=\underline{\operatorname{Sub}}\,{Q_k}$, so $\underline{\CM}(A)$ is $2$-Calabi--Yau, and $\CM(A)$ is a Frobenius cluster category. 

Unlike $\operatorname{Sub}\, Q_{k}$ and the $\curly{C}_w$, the category $\CM(A)$ is not Hom-finite. However, as already observed, the endomorphism algebras of its objects are Noetherian, so we may apply our general theory to this example.

In their study of the category $\CM(A)$, Jensen, King and Su show the following.  Let
\[ \integ^{n}(k)=\{ x\in \integ^{n} \mid k\ \text{divides} \textstyle\sum_{i} x_{i} \} \] with basis $\alpha_{1},\dotsc ,\alpha_{n-1},\beta_{[n]}$, where the $\alpha_{j}=e_{j+1}-e_j$ are the negative simple roots for $\mathrm{GL}_{n}(\complex)$ and $\beta_{[n]}=e_{1}+\dotsm +e_{k}$ is the highest weight for the representation $\bigwedge^{k}(\complex^{n})$.  

Then by \cite[\S 8]{JKS} we have that $\mathrm{K}_{0}(\CM (A))\iso \mathrm{K}_{0}(A)\iso \integ^{n}(k)$; let $G\colon \mathrm{K}_{0}(\CM(A))\to \integ^{n}(k)$ denote the composition of these isomorphisms.  The $\mathrm{GL}_{n}(\complex)$-weight of the cluster character of $M\in \CM(A)$ (called $\tilde{\psi}_{M}$ in $\cite{JKS}$) is given by the coefficients in an expression for $G[M]\in\integ^{n}(k)$ in terms of the basis of $\integ^n(k)$ given above \cite[Proposition~9.3]{JKS}, and thus this weight defines a group homomorphism $\mathrm{K}_0(\CM(A))\to\integ^n$.

Said in the language of this paper, $\CM(A)$ is a graded Frobenius cluster category with respect to $\mathrm{GL}_{n}(\complex)$-weight, this giving a standard integer multi-grading.

Let $\delta\colon \integ^{n}(k)\to \integ$ be the (linear) function $\delta(x)=\frac{1}{k}\sum_{i} x_{i}$.  By the linearity of gradings, composing $G$ with $\delta$ yields a $\integ$-grading on $\CM (A)$ also.  Explicitly, $\delta(x)$ is the $\beta_{[n]}$-coefficient of $x$ in our chosen basis, and is also equal to the dimension of the socle of $\pi M$, which is equal to $\mathrm{rk}(M)$, which is equal to the degree of the cluster character of $M\in \CM(A)$ as a homogeneous polynomial in the Pl\"{u}cker coordinates of the Grassmannian.  

It is well known that the cluster structure on the Grassmannian is graded with respect to either the $\mathrm{GL}_{n}(\complex)$-weight (also called the content of a minor, and, by extension, of a product of minors) or the natural grading associated to the Pl\"{u}cker embedding.  The results of \cite{JKS} show that these gradings are indeed naturally reflected in the categorification of that cluster structure.  This opens the possibility of attacking some questions on, for example, the number of cluster variables of a given degree by examining rigid indecomposable modules in $\CM(A)$ of the corresponding rank, say.  We hope to return to this application in the future.

Of course, one can also argue directly that $\mathrm{rk}(M)$ yields a grading on $\CM (A)$, considering it as a function on $\mathrm{K}_{0}(\CM(A))$.  Note that the socle dimension of $\pi M$ is not a grading on $\operatorname{Sub}\, Q_{k}$, but rather it is the datum within $\operatorname{Sub}\, Q_{k}$ that specifies how one should lift $\pi M$ to $M$ (see \cite[\S 2]{JKS} for an illustration of this).  As described in the previous section, $\operatorname{Sub}\, Q_{k}$ (in its guise as one of the $\curly{C}_{w}$) does admit gradings, such as the grading describing the degree of the cluster character of $\pi M\in \operatorname{Sub}\, Q_{k}$ (called $\psi_{\pi M}$ in \cite{JKS}) with respect to the standard matrix generators.

\small

\bibliographystyle{halpha}
\bibliography{biblio}\label{references}

\normalsize

\end{document}